\documentclass{amsart}
\usepackage{amsmath,amssymb}
\usepackage{pgf,tikz}
\def\Z{\mathbb{Z}}

\def\R{\mathbb{R}}

\def\F{\mathcal{F}}
\newcommand{\C}{\mathcal{C}}
\newcommand{\rel}{\mathcal{R}}

\newcommand{\N}{\mathbb{N}}
\newcommand{\p}{\mathcal{P}}

\newcounter{thm}
\newcounter{ex}
\newcounter{re}
\newtheorem{Theorem}[thm]{Theorem}
\newtheorem{Lemma}[thm]{Lemma}

\newtheorem{Proposition}[thm]{Proposition}
\newtheorem{example}[thm]{Example}
\newtheorem{remark}[thm]{Remark}
\newtheorem{Definition}[thm]{Definition}

\title[Smooth infinite dimensional cross-ratio]{On smooth infinite dimensional grassmannians, splittings and non-commutative generalized cross-ratio mappings}
\author{Jean-Pierre Magnot }
\address{J.-P.M.: Univ. Angers, CNRS, LAREMA, SFR MATHSTIC, F-49000 Angers, France
	\\ and \\  Lyc\'ee Jeanne d'Arc \\ Avenue de Grande Bretagne, \\ 63000 Clermont-Ferrand, France}
\email{magnot@math.cnrs.fr}

\begin{document}

\begin{abstract}
We describe basic diffeological structures related to splittings and Grassmannians for infinite dimensional vector spaces. We analyze and expand the notion of non-commutative cross-ratio and prove its smoothness. Then we illustrate this theory by examples, with some of them extracted from the existing literature related to infinite dimensional (Banach) Grassmannians, and others where the diffeological setting is a key primary step for rigorous definitions.     
\end{abstract}

\maketitle


\smallskip

\smallskip


\section{Introduction}
The notion of cross-ratio plays an important role in modern geometry, and its basic application to characterization of cocyclic points in the plane has been for a while at the curriculum of high school students in some EU states. A clear exposition of its importance can be found in \cite{Lab2008}. Some various generalizations in symplectic and non-commutative geometry arise in recent algebraic and geometric works, see e.g. \cite{CO2019,DGP2013,GGRW2005,LMcS2009,Ret2014,RRS2020,Ze2006}. We analyze here one non-commutative analog to the cross-ratio present, algebraically, in a remark of \cite{RRS2020} following the pioneering work \cite{Ze2006}. We show that there are non necessary hypothesis in its definition, which gathers previously described compatible approaches. Therefore, we refresh the algebraic approach already highlighted in a \emph{smooth} or \emph{geometric} way. Our enhanced definition of non commutative cross-ratio between two admissible splittings is shown to produce a vector space isomorphism which depends \emph{smoothly} on the underlying admissible couple of splittings. Then we build up remarkable involutions on the space of admissible couples splittings, and derive from these involutions and from our initial non-commutative cross-ratio mapping that we note by $DV,$ two other mappings, that we note by $\tilde{DV}$ and $\Xi,$ from admissible couples of splittings to $GL(V).$ For this, a safe notion of smoothness on Grassmanians, splittings, and on the group $GL(V)$ is necessary.     

The geometry of finite dimensional Grassmanians in a finite dimensional vector space is an underlying object for many methods in differential geometry. Its generalizations to infinite dimensional settings are, in the actual state of knowledge, subject  to technical restrictions, and the same holds for the geometric structures that can be developed on Grassmanian ``manifolds" in view of applications. In most examples, the Grassmanian space of an infinite dimensional vector space $V$ is constituted by finite dimensional and finite codimensional vector subspaces of $V$, or it is defined through a fixed splitting $V = V_+ \oplus V_-  $ by two closed supplementary infinite dimensional vector spaces of $V.$ In each of these cases, this is obviously only one part of the full space of closed vector subspaces (with topological supplement) of $V$ that is under consideration. In particular, the important example of the restricted Grassmanian defined in the works of Pressley, Segal and Wilson (see e.g. \cite{Mick,PS} for two expositary textbooks on this subject) is defined through the eigenspaces of the sign of the Dirac operator acting on $L^2(S^1,\mathbb{C}),$ partly inspired by Sato's geometric approaches on integrable systems. These initial works had ramifications until very recent works, in related but different subjects. A selected bibliography on these developments is \cite{ACD2014,ACL2018,ALR2010,BRT2007,CDMP,CJ2019,GGRW2005,GO2010,Ma2006,Mick2,RRS2020,SJ2019,Tu2020}.

In all these works, one main technical limit to the investigations is the necessary underlying Hilbert or Banach space topology on the vector space $V,$ that enables, using smart restrictions, to define an atlas on the space of so-called restricted Grassmanians. We propose here to describe a differential geometry in the \emph{generalized} setting of diffeologies. We study spaces of splittings of an infinite dimensional vector space that can be in the category of Hilbert, Banach, Fr\'echet, locally convex, or even more generally, Fr\"olicher or diffeological vector spaces. Splittings are understood as eigenspaces of linear symmetries of $V.$ Within diffeological considerations, technical difficulties vanish to define a well-suited smoothness on spaces of splittings, and then on the space of Grassmanians constituted by  vector subspaces (closed, with supplement). 

The same way, the same diffeologies enable to talk safely of smooth linear operators acting on an infinite dimensional vector space $V$ which is not in the Hilbert or Banach setting, and to define smoothness on $L(V)$ and $GL(V).$ Therefore, this is natural for us to start this paper by necessary statements on diffeologies, known and new (section \ref{s:diff}). In particular, a short exposition on diffeological vector spaces, diffeological algebras and their groups of the units, and on the particular case of the linear algebra $L(V)$ and the linear group $GL(V)$ of a diffeological vector space $V,$ completes the existing statements in the literature. Therefore, we have all the necessary technical tools to define, given a diffeological algebra $A,$ the space of splittings of $V$ from the space $S(A)$ of involutive elements (or symmetries, section \ref{s:3.1}) of $A.$ Then the non-commutative cross-ratio $DV$ already announced can be defined on \emph{admissible} couples of splittings $Adm(V)$ in section \ref{s:3.2} and is easily proven to be smooth, with values in a linear group. From these basic descriptions, we highlight involutions on $\mathbf{S}(A)$ and on $Adm(V)$ that are intrinsically linked with the construction of the non-commutative cross-ratio. Then we extend $DV$ to two $GL(V)-$valued maps in section \ref{s:3.4} and in section \ref{s:3.5}. The first one, $\tilde{DV},$ is a direct extension to the cross-ratio map. The second one, noted by $\Xi,$ defines a smooth diffeomorphism on the space of topological supplement to a vector subspace of $V.$ We call this diffeological space the Grassmanian relative to $V.$ After these abstract constructions, this is highly needed to develop examples in section \ref{s:4}. The Pressley-Segal-Wilson restricted Grassmanian gives us a first framework in section \ref{s:4.1} that enables to highlight basic geometric properties of admissible splittings as they are described in our approach. But this setting is not adapted for a simple calculation of a non-trivial non-commutative cross-ratio in the actual state of knowledge and in the context of an illustration of this work. In order to produce a non trivial cross-ratio in an accessible and straightforward computation, we produce in section \ref{s:4.2.2} a second example of admissible splittings in the algebra of formal classical pseudo-differential operators which can be understood as a locally convex non complete vector space. The admissible couple of splittings that we consider is derived, from one side, from Kontsevich and Vishik's works on spectral analysis \cite{KV1}, see e.g. \cite{Scott}, and also from the Pressley-Segal-Wilson polarization. In this example, the non-commutative cross-ratio $\tilde{DV}$ is proven to be non-trivial. Finally, we consider splittings on spaces of mappings defined by signed measures. We prove that two different signed measures do not produce any admissible couple of splittings in the context that we consider.

These three settings intend to produce examples and to give illustrations to the complex constructions that are developed. The main novelty of this paper is, to our own feeling, on one side the smooth structures on Grassmanians and splittings that are no longer technically restrained, and on the other side the maps $DV,$ $\tilde{DV}$ and $\Xi,$ that extend the (commutative) cross-ratio in non-commutative setting, are now safely defined on admissible couples of splittings. Moreover, two among the three of our testing examples (formal pseudo-differential operators, signed measures) could not be reached in a Banach manifold category. Our approach has now to be compared to the other existing ones. 

	\section{A presentation of diffeologies and Fr\"olicher spaces} \label{s:diff}

\subsection{Diffeological and Fr\"olicher spaces}
In this section we follow
the expositions appearing in \cite{Ma2013,MR2016,MR2019}, and complete the classical exposition \cite{Igdiff} along the lines of the recent review \cite{GMW2023}. 

\begin{Definition} Let $X$ be a set.
	
	\noindent $\bullet$ A \textbf{p-parametrization} of dimension $p$ 
	on $X$ is a map from an open subset $O$ of $\R^{p}$ to $X$.
	
	\noindent $\bullet$ A \textbf{diffeology} on $X$ is a set $\p$
	of parametrizations on $X$ such that:
	
	- For each $p\in\N$, any constant map $\R^{p}\rightarrow X$ is in $\p$;
	
	- For any set of indexes $I$ and any family $\{f_{i}:O_{i}\rightarrow X\}_{i\in I}$
	of compatible maps that extend to a map $f:\bigcup_{i\in I}O_{i}\rightarrow X$,
	if $\{f_{i}:O_{i}\rightarrow X\}_{i\in I}\subset\p$, then $f\in\p$.
	
	- $\forall f\in\p$, $f : O\subset\R^{p} \rightarrow X$, $g \forall g \in C^\infty(O' , O)$, 
	in which  $O'$ is an open subset of an Euclidean space, we have $f\circ g\in\p$.
	
	\vskip 6pt  $(X,\p)$ is called if 
	called a \textbf{diffeological space} $\p$ is a diffeology on $X$ and, if $(X,\p)$ and $(X',\p')$ are two diffeological spaces, 
	a map $f:X\rightarrow X'$ is \textbf{smooth} if and only if $f\circ\p\subset\p'$. 
\end{Definition} 

For an historical introduction to diffeologies, see \cite{Sou}, inspired by the remarks from Chen in \cite{Chen}. The textbook \cite{Igdiff} gives a comprehensive exposition of basic concepts. 
\begin{Definition}
	Let $(X,\p)$ be a diffeological space. Let $\p'$ be a diffeology on $X.$ Then $\p'$ is a \textbf{subdiffeology} of $\p$ if and only if $\p' \subset \p$ or, equivalently, $Id_X:(X,\p')\rightarrow (X,\p)$ is smooth.
\end{Definition}
\begin{example}
	Let $X$ be a nonempty set. There is a minimal diffeology for inclusion. This diffeology is made of constant parametrizations and is called the \textbf{discrete diffeology}.
\end{example}
The category of diffeological spaces is very large, and carries many different pathological examples. Therefore, a restricted category can be useful, the category of Fr\"olicher spaces.

\begin{Definition} 
	A \textbf{Fr\"olicher} space is a triple $(X,\F,\C)$ such that
	
	- $\C$ is a set of paths $\R\rightarrow X$, called set of \textbf{contours}.
	
	- $\F$ is the set of functions from $X$ to $\R$, such that a function
	$f:X\rightarrow\R$ is in $\F$ if and only if for any
	$c\in\C$, $f\circ c\in C^{\infty}(\R,\R)$;
	
	- Let $c:\R\rightarrow X.$ $c \in \C \, \Leftrightarrow \, \forall f\in\F$, $f\circ c\in C^{\infty}(\R,\R)$.
	
	\vskip 5pt If $(X,\F,\C)$ and $(X',\F',\C ')$ are two
	Fr\"olicher spaces, a map $f:X\rightarrow X'$ is \textbf{smooth}
	if and only if $\F'\circ f\circ\C\subset C^{\infty}(\R,\R)$.
\end{Definition}

\noindent
This definition first appeared in \cite{FK} but the actual terminology 
was fixed to our knowledge in \cite{KM}.
The comparison of these two frameworks was published in \cite{Ma2006-3} in its first steps, the reader can
also see \cite{BKW2024,Ma2013,Ma2018-2,MR2016,Wa} for extended expositions.
In particular, it is explained in \cite{MR2016} that 
{\em Diffeological, Fr\"olicher and G\^ateaux smoothness are the same notion if we 
	restrict ourselves to a Fr\'echet context.}  For this, let us analyze how we generate a Fr\"olicher space, that is, how we implement a Fr\"olicher  structures on a given set $X.$
Any family of maps $\F_{g}$ from $X$ to $\R$ generates a 
Fr\"olicher structure $(X,\F,\C)$ by setting, after \cite{KM}:

- $\C=\{c:\R\rightarrow X\hbox{ such that }\F_{g}\circ c\subset C^{\infty}(\R,\R)\}$

- $\F=\{f:X\rightarrow\R\hbox{ such that }f\circ\C\subset C^{\infty}(\R,\R)\}.$

\noindent
We call $\F_g$ a \textbf{generating set of functions}
for the Fr\"olicher structure $(X,\F,\C)$.
A Fr\"olicher space $(X,\F,\C)$
carries a natural topology, the pull-back topology of
$\R$ via $\F$. 
Alternatively, one can start alternatively from a generating set of contours $\C_g,$ and set:

- $\F=\{f:X\rightarrow\R\hbox{ such that }f\circ\C_g\subset C^{\infty}(\R,\R)\}$

- $\C=\{c:\R\rightarrow X\hbox{ such that }\F\circ c\subset C^{\infty}(\R,\R)\}.$

In the case of a finite dimensional
differentiable manifold $X$ we can take $\F$ as the set of all smooth
maps from $X$ to $\R$, and $\C$ the set of all smooth paths from
$\R$ to $X.$ Then, the underlying topology of the
Fr\"olicher structure is the same as the manifold topology
\cite{KM}. We also remark that if $(X,\F, \C)$ is a Fr\"olicher space, we can
define a natural diffeology on $X$ by using the following family
of maps $f$ defined on open domains $D(f)$ of Euclidean spaces, see \cite{Ma2006-3}:
$$
\p_\infty(\F)=
\coprod_{p\in\N^*}\{\, f: D(f) \rightarrow X\, | \, D(f) \hbox{ is open in } \R^p \hbox{ and } \F \circ f \in C^\infty(D(f),\R) \}.$$
If $X$ is a finite-dimensional differentiable manifold, setting $\F=C^\infty(X,\R),$ this diffeology is
called the { \em nebulae diffeology}, see e.g. \cite{Igdiff}. We extend this diffeology to any Fr\"olicher space. Now,
we can easily show the following:

\begin{Proposition} \label{fd} \cite{Ma2006-3}
	Let$(X,\F,\C)$
	and $(X',\F',\C')$ be two Fr\"olicher spaces. A map $f:X\rightarrow X'$
	is smooth in the sense of Fr\"olicher if and only if it is smooth for
	the underlying diffeologies $\p_\infty(\F)$ and $\p_\infty(\F').$
\end{Proposition}

Thus, Proposition \ref{fd} and the foregoing remarks imply that 
the following implications hold:
\vskip 12pt

\begin{tabular}{ccccc}
	smooth manifold  & $\Rightarrow$  & Fr\"olicher space  & $\Rightarrow$  & diffeological space
\end{tabular}

	\subsection{Diffeologies on spaces of mappings, products, quotients, subsets and algebraic structures}
	\begin{Proposition} \label{prod1} \cite{Igdiff,Sou} Let $(X,\p)$ and $(X',\p')$
		be two diffeological spaces. There exists a diffeology 
		$\p\times\p'$ on
		$X\times X'$  made of plots $g:O\rightarrow X\times X'$
		that decompose as $g=f\times f'$, where $f:O\rightarrow X\in\p$
		and $f':O\rightarrow X'\in\p'$. We call it the \textbf{product diffeology}, 
		and  this construction extends to an infinite (maybe not countable) product.
	\end{Proposition}

	The same holds on contours of Frölicher spaces in a cartesian product (compare with \cite{KM}).
	%
	  %
	We can also state the above result for infinite products;
	we simply take Cartesian products of the plots, or of the contours.

	After \cite{Sou} and 
	\cite[p. 27]{Igdiff}, let $(X,\p)$ be 
	a diffeological space, and let $X'$ be a set. Let 
	$f:X\rightarrow X'$ be a map. We define the \textbf{push-forward 
		diffeology}, noted by $f_*(\p),$ as the coarsest (i.e. minimal for inclusion) among 
	the diffologies on $X'$, which contains $f \circ \p,$ that is, for which $f$ is smooth. Conversly, if $(X',\p')$ is a diffeological space and if $X$ is a set, if $f:X\rightarrow X',$ the \textbf{pull-back diffeology} noted by $f^*(\p')$ is the diffeology on $X,$ maximal for inclusion, for which $f$ is smooth. 
	
	\begin{Proposition} \label{quotient} \cite{Igdiff,Sou}
		Let $(X,\p)$ b a diffeological space and $\rel$ an equivalence 
		relation on $X$. There is a natural diffeology on $X/\rel$, 
		noted by $\p/\rel$, defined as the push-forward diffeology on 
		$X/\rel$ by the quotient projection $X\rightarrow X/\rel$. This diffeology is made of plots that read locally as maps $\pi \circ p,$ where $p \in \p.$
	\end{Proposition}
	
	Given a subset $X_{0}\subset X$, where $X$ is a Fr\"olicher space
	or a diffeological space, we equip $X_{0}$ with structures 
	induced by $X$ as follows:
	\begin{enumerate}
		\item If $X$ is equipped with a diffeology $\p$, we define
		a diffeology $\p_{0}$ on $X_{0}$ called the \textbf{subset or trace 
			diffeology}, see \cite{Sou,Igdiff}, by setting 
		\[ 
		\p_{0}=
		\lbrace p\in\p \hbox{ such that the image of }p\hbox{ is a subset 
			of } X_{0}\rbrace\; .
		\]
		\item If $(X,\F,\C)$ is a Fr\"olicher space, we take as a 
		generating set of maps $\F_{g}$ on $X_{0}$ the restrictions of the 
		maps $f\in\F$. In this case, the contours (resp. the induced 
		diffeology) on $X_{0}$ are the contours (resp. the plots) on $X$ 
		whose images are a subset of $X_{0}$.
	\end{enumerate}
	\begin{example}
		Let us consider the infinite product diffeology on $Y^X.$ This is the largest diffeology for which the evaluation maps $ev_x$ are smooth for each $x \in X.$ Therefore, any other diffeology on any subset of $Y^X$ for which all the evaluation maps are smooth is a subdiffeology of the subset diffeology inherited from this product diffeology.
	\end{example}
	Following \cite{Igdiff}, 
	let $(X,\p)$ and $(X',\p')$ be diffeological spaces. 
	Let $M \subset C^\infty(X,X')$ be a set of smooth maps. 
	The \textbf{functional diffeology} on $S$ is the diffeology $\p_S$
	made of plots 
	$$ \rho : D(\rho) \subset \R^k \rightarrow S$$
	such that, for each $p \in \p$, the maps 
	$\Phi_{\rho, p}: (x,y) \in D(p)\times D(\rho) \mapsto 
	\rho(y)(x) \in X'$ are plots of $\p'.$ 
	We have, see \cite[Paragraph 1.60]{Igdiff}:
	
	\begin{Proposition} \label{cvar} 
		Let $X,Y,Z$ be diffeological spaces. Then, 
		$$
		C^\infty(X\times Y,Z) = C^\infty(X,C^\infty(Y,Z)) = 
		C^\infty(Y,C^\infty(X,Z))
		$$
		as diffeological spaces equipped with functional diffeologies.
	\end{Proposition}

	
	
	
	
	\subsection{Groups, vector spaces, algebra and fiber bundles in diffeologies}
	Given an algebraic structure, we can define a
	corresponding compatible diffeological (resp. Fr\"olicher) structure, see 
	for instance \cite{Les}. For example, see \cite[pp. 66-68]{Igdiff},
	if $\R$ is equipped with its canonical diffeology (resp. Fr\"olicher
	structure), we say that an $\R-$vector space equipped with a 
	diffeology (resp. Fr\"olicher structure) is a diffeological (resp. Fr\"olicher) vector space if addition and 
	scalar multiplication are smooth. We state the same way, following \cite{Les,Ma2013}:

	\begin{Definition}
		Let $G$ be a group equipped with a diffeology (resp. Fr\"olicher 
		structure). We call it a \textbf{diffeological (resp. Fr\"olicher) group} 
		if both multiplication and inversion are smooth.
	\end{Definition}

 \subsubsection{The group of diffeomorphisms of a diffeological space.}\label{ss:diff}
 Let $(X,\p)$ be a diffeological space. The space $C^\infty(X,X)$ is endowed with its functional diffeology and a diffeomorphism on $X$ is a map $g \in C^\infty(X,X)$ such that there exists an inverse map $g^{-1} \in C^\infty(X,X)$ such that $g \circ g^{-1} = g^{-1} \circ g =Id_X.$ The elementary algebraic properties of the inverse assert that the inverse is unique. The set of diffeomorphisms on $X$ is noted by $Diff(X).$ Unlike for the Fr\'echet Lie group of diffeomorphisms of a finite dimensional compact manifold, the group $Diff(X)$ is not a priori a diffeological group because the inverse map $g \mapsto g^{-1}$ is not a priori smooth for the (trace of) the functional diffeology $\p_{Diff(X)}$ on $Diff(X).$ In fact, only the composition of diffeomorphisms is smooth. Our unability to state smoothness of the inversion is due to a lack of inverse functions theorem. Therefore, we have to define the following diffeology. 

 \begin{Proposition}
     Let $\p_{Diff(X)}^{(-1)}$ be the diffeology on $Diff(X)$ defined by pull-back $$\p_{Diff(X)}^{(-1)} = \Phi^*\left(\p_{Diff(X)} \times \p_{Diff(X)}\right)$$
     where $$\Phi: g \in Diff(X) \mapsto (g,g^{-1}) \in Diff(X)\times Diff(X).$$
     Then $\left(Diff(X),\p_{Diff(X)}^{(-1)}\right)$ is a diffeological group.
 \end{Proposition}

 The proof is straightforward. Therefore, considering now a diffeological ring $R$, the set of invertible elements $R^*$ of $R$ has itself a specific diffeology $\p_{R^*}^{(-1)}$ which differ from the subset diffeology $\p_{R^*}$ inherited from the inclusion $R^* \subset R.$ We only have in general $$\p_{R^*}^{(-1)} \subset \p_{R^*}.$$ 
 The same construction holds for a diffeological field, a diffeological module, a diffeological algebra: the group of the units (i.e. the group of the invertible elements) can be endowed with a diffeology ``òf $\p^{(-1)}-$type''. 
 
 \begin{remark} \label{rk:units}
     In the case of a $\R-$algebra or  a $\mathbb{C}-$algebras $A$, the group $A^*$ may carry this pathology and there is an identified class of examples of infinite dimensional algebras for which $A^*$ is a Lie group for the subset diffeology, in the setting of locally convex, Mackey-complete algebras \cite{Glo2002,GN2012}. Particular examples of this class are Banach algebras. In that case, $\p_{A^*}^{(-1)} = \p_{A^*}.$ 
 \end{remark} 
 
 \begin{example}
     When considering fields, the groups $\R^*$ and $\mathbb{C}^*$ have inversion which is smooth for the subset diffeology, in other words, $\p_{ \R^*}^{(-1)} = \p_{\R^*}$ and $\p_{\mathbb{C}^*}^{(-1)} = \p_{\mathbb{C}^*}.$
 \end{example}
\subsubsection{Linear group of a diffeological vector space} \label{ss:l(v)}
Let $V,V'$ be two diffeological vector spaces. We note by $L(V,V')$ the space of smooth linear maps from $V$ to $V'$ equipped with the  functional diffeology (understood as a subset diffeology inherited from the inclusion $L(V;V') \subset C^\infty(V,V')$). We note $L(V)$ for short, instead of $L(V,V).$  

We recall that, with no local inverse functions theorem at hand, it is impossible to prove that the group $GL(V)$ of invertible maps has a smooth inversion, and following \cite{GMW2023}, we equip the (algebraic) group $GL(V)$ with the diffeology of $Diff(V)$ (understood as the subset diffeology inherited by the inclusion $GL(V)\subset Diff(V)$). Since composition and inversion is smooth on $Diff(V),$ so it is on $GL(V).$  This construction provides a concrete class of examples of diffeological algebras and their diffeological group of the units  that complete the classical examples recalled in Remark \ref{rk:units} and where the diffeology $\p_{A^*}^{(-1)}$ may differ from the diffeology $\p_{A^*}.$ These examples carry a wide range  of linear algebras of vector spaces with adequate properties to define smoothness, including vector spaces with pathological properties, see e.g. \cite{perv-grenoble,wu-grenoble} for actual concerns on ``exotic'' diffeological vector spaces and their technical properties. 
\section{Projections, symmetries, Grassmannians, splittings and non-commutative cross-ratio}
 \subsection{Symmetries and projections} \label{s:3.1}
 Let $A$ be a diffeological unital algebra with group of the units $A^*.$  This algebra is assumed to act smoothly on a diffeological vector space $V.$ 
 We note by $$\mathbf{P}(A) = \left\{ p \in A \, | \, p^2 = p \right\}$$
 the set of projectors in $A$ and by 
 $$ \mathbf{S}(A) = \left\{ s \in A \, | \, \exists p \in \mathbf{P}(A), s = 2p - 1_A \right\}.$$ 
 \begin{remark}
     The classical condition $s^2=1_A$ is equivalent, that is, $$\forall x \in V, x \mapsto \frac{1}{2}(x+s(x))$$ defines the projection $p \in \mathbf{P}(A)$ such that $s = 2p - 1_A .$
 \end{remark}
 When $A = L(V),$  this picture specialize to the following.
 We note by $$\mathbf{P}(V) = \left\{ p \in L(V) \, | \, p^2 = p \right\}$$
 the set of (diffeological) projectors in $V$ and by 
 $$ \mathbf{S}(V) = \left\{ s \in L(V) \, | \, s^2 = Id_V \right\}$$ the set of (diffeological) symmetries on $V.$
 \begin{Lemma} 
 	 Let $i_A$ and $i_{A^*}$ be the inclusion maps from $\mathbf{S}(A)$ to $A$ and to $A^*$ respectively, then $i_A^*(\p_A) = i_{A^*}^*(\p_{A^*})=i_{A^*}^*(\p_{A^*}^{(-1)}).$
 \end{Lemma} 
\begin{proof}
	The map 
	$$\Phi: u \in A^* \mapsto (u,u^{-1})\in A^2$$
	 restricts to the canonical (diagonal) inclusion $u \mapsto (u,u)$ from $\mathbf{S}(A)$ to $A^2.$
	\end{proof}
We also note by $\p_A$ the subset diffeology of $\mathbf{P}(A) \subset A$  and as well as the subset diffeology of $\mathbf{S}(A) \subset A.$ We can state:
\begin{Proposition}
	There is a diffeomorphism
	$$ p \in \mathbf{P}(A) \mapsto 2 p - 1_A  \in \mathbf{S}(A)$$
	for the diffeology $\p_A$ on $\mathbf{P}(A).$
\end{Proposition}
\begin{proof}
	It follows from the smoothness of addition and multiplication on $A.$
	\end{proof}
The previous results apply to $\mathbf{P}(V) $ and to
$ \mathbf{S}(V) $ according to the description of $L(V)$ as a diffeological algebra given in section \ref{ss:l(v)}.
Let us recall, following the classical algebraic decompositions, that given $s \in \mathbf{S}(A) $ and $u \in A,$ 
\begin{eqnarray} \label{eq:dec-s} 
(s+u) \in \mathbf{S}(A) & \Leftrightarrow & su+us+u^2=0
\end{eqnarray}  
Let $$AS(s) : u \in A \mapsto us+su \in A.$$
Let $s \in \mathbf{S}(A).$ We define
$$ \pi_s^T: a \in A \mapsto \frac{1}{2} \left(a - s a s\right)  $$
and 
$$ \pi_s^N: a \in A \mapsto \frac{1}{2} \left(a + s a s\right)$$
Let us now give an easy but useful lemma:
\begin{Lemma} \label{Lemma-pi-split-1}
	$\forall s \in \mathbf{S}(A),$ 
	\begin{itemize}
		\item $\pi_s^T \in \mathbf{P}(L(A))$ and $\pi_s^N \in \mathbf{P}(L(A))$ 
		\item $Im (\pi_s^T) = Ker (\pi_s^N)$ and $Im (\pi_s^N) = Ker (\pi_s^T)$
		\item $Im (\pi_s^T) = Ker AS(s)$
		\item $A$ is isomorphic, as a diffeological vector space, with  $$Im (\pi_s^T) \oplus Im (\pi_s^N)$$ where each component is equipped with the subset diffeology. In other words, each $a\in A$ splits uniquely into $a = a_s^T + a_s^N,$ with  $ a_s^T \in Im (\pi_s^T),$ $ a_s^N \in Im (\pi_s^N),$ and the mapping $a \mapsto (a_s^T,a_s^N) = \left(\pi_s^T(a),\pi_s^N(a)\right)$ is a diffeomorphism.
	\end{itemize} 
\end{Lemma}
\begin{proof}
	In the results that have to be proven, all the algebraic parts are elementary and well-established. 
	
	For the diffeological parts, it follows from the smoothness of scalar multiplication, addition and multiplication in $A.$
\end{proof}
Let us now analyze the interaction between the splitting $a = a_s^T + a_s^N$ with the multiplication. 
We calculate, for $s\in \mathbf{S}(A)$ and for $(a,b) \in A^2,$
\begin{eqnarray*}
	a_s^T b_s^T  & = & \frac{1}{4}\left(a - s a s\right)\left(b - s b s\right)\\
	& = & \frac{1}{2}\pi_s^N(ab) - \frac{1}{4}\left(sasb+asbs\right),
	\end{eqnarray*}
\begin{eqnarray*}
	a_s^N b_s^N  & = & \frac{1}{4}\left(a + s a s\right)\left(b + s b s\right)\\
	& = & \frac{1}{2}\pi_s^N(ab) + \frac{1}{4}\left(sasb+asbs\right),
\end{eqnarray*}
\begin{eqnarray*}
	a_s^T b_s^N  & = & \frac{1}{4}\left(a - s a s\right)\left(b + s b s\right)\\
	& = & \frac{1}{2}\pi_s^T(ab) - \frac{1}{4}\left(sasb-asbs\right),
\end{eqnarray*}
\begin{eqnarray*}
	a_s^N b_s^T  & = & \frac{1}{4}\left(a + s a s\right)\left(b - s b s\right)\\
	& = & \frac{1}{2}\pi_s^T(ab) + \frac{1}{4}\left(sasb-asbs\right).
\end{eqnarray*}

	These are all intermediate results which have as a consequence that $ab = (ab)_s^T + (ab)_s^N$ but we can now precise a ``polarization like'' property which is a straightforward consequence of our computations:
	
	\begin{Lemma}
		Let $s\in \mathbf{S}(A)$ and let $(a,b) \in A^2.$
		$$ (ab)_s^T = a_s^T b_s^N + a_s^N b_s^T$$
		and 
		$$(ab)_s^N = a_s^T b_s^T + a_s^N b_s^N.$$
	\end{Lemma}

Next theorem follows straightway from our preliminary constructions.
\begin{Theorem}
	for each $s \in \mathbf{S}(A),$, we define a $\mathbb{Z}_2$-grading on $A = Im (\pi_s^T) \oplus Im (\pi_s^N)$ by assigning $deg_s(a_s^N) = 1$ and $deg_s(a_s^T)=-1,$ and the algebra $A$ reads as a diffeological $\mathbb{Z}_2-$graded algebra with respect to the decomposition $A = Im (\pi_s^N) \oplus Im (\pi_s^T).$ 
	
	Moreover, the map $$(a,s)\in A \times \mathbf{S}(A) \mapsto (a_s^N, a_s^T) \in A^2$$ is smooth. 
\end{Theorem}
	
\begin{remark}
	$\forall s \in \mathbf{S}(A), s \in Im (\pi_s^N),$ $Id_A \in Im (\pi_s^N)$ and $ p = \frac{1}{2}(s+ Id_A) \in Im (\pi_s^N).$
\end{remark}
Let us now investigate the following problem: let $a \in A.$ For which $s \in \mathbf{S}(A)$ do we have 
$a \in Im\left(\pi_s^N\right)?$ or $a \in Im\left(\pi_s^N\right)?$

\begin{Proposition}
	Let $a \in A.$ Then $a \in Im\left(\pi_s^T\right)$ if and only if
	\begin{equation}
	\label{eq:tangent}
	(a+s)^2 = a^2 + Id_A
	\end{equation}
\end{Proposition}
\begin{proof} $a \in Im\left(\pi_s^T\right) \leftrightarrow a \in Ker\left(\pi_s^N\right) \leftrightarrow as + sa =0. $ 
	Then expanding
	\begin{eqnarray*}
		(a+s)^2 - a^2 - Id & = & as + sa
		\end{eqnarray*}
	we get the result.
	\end{proof}
\begin{Proposition}
	Let $a \in A.$ Then $a \in Im\left(\pi_s^N\right)$ if and only if $$[a,s]=0$$
	which is equivalent to 
	\begin{equation}
	\label{eq:normal}
	(a+s)(a-s) = a^2 - Id_A
	\end{equation}
	or 
	\begin{equation}
	\label{eq:normal-2}
	(a+s)(a-s) = \left(a - Id_A\right)\left(a + Id_A\right)
	\end{equation}
\end{Proposition}
\begin{proof} $a \in Im\left(\pi_s^N\right) \leftrightarrow a \in Ker\left(\pi_s^T\right) \leftrightarrow as - sa =0. $ 
	Then expanding
	\begin{eqnarray*}
		(a+s)(a-s) - a^2 + Id_A & = & as + sa
	\end{eqnarray*}
	we get the result, since $[a,Id_A]=0 \Leftrightarrow a^2-Id_A = (a-Id_A)(a+Id_A).$
\end{proof}
\subsection{From splittings to Grassmannians and to projections and non-commutative cross-ratio} \label{s:3.2}
\begin{Definition}
	The space of (diffeological) splittings of $A$ is the space 
		$$\mathbf{Spl}(A) = \left\{ (V,W) \subset A^2 \, | \, \exists s \in \mathbf{S}(A), \, V = Ker(s-Id) \wedge W = Ker(s+Id)\right\}.$$
		\end{Definition}
	\begin{Proposition}
		We have equivalently
		$$\mathbf{Spl}(A) = \left\{ (V,W) \subset A^2 \, | \, \exists p \in \mathbf{P}(A), \, V = Im(p) \wedge W = Ker(p)\right\}.$$
		We equip $\mathbf{Spl}(A)$ with its push-forward diffeology $\p_\mathbf{Spl} = (Im \times Ker)_* \p_L.$
\end{Proposition}

We define the map $J_{(-)}: a \in A \mapsto -a \in A.$ 
Here are few easy properties:
\begin{Proposition}
	\begin{enumerate}
		\item The map $Im \times Ker: \mathbf{P}(A) \rightarrow \mathbf{Spl}(A)$ is a diffeomorphism, and induces a diffeomorphism $Pol:   \mathbf{S}(A) \rightarrow \mathbf{Spl}(A)$.
		\item $\forall (V,W) \in \mathbf{Spl}(A),$ $$ J_{(-)}^{\mathbf{Spl}} (V,W) = Pol \circ J_{(-) }\circ Pol^{-1}(V,W)=(W,V).$$ This enables to define a smooth involution on $\mathbf{Spl}(A)$ that we also note by $J_{(-)}$ instead of $J_{(-)}^{\mathbf{Spl}}$ when it carries no ambiguity.
	\end{enumerate}
\end{Proposition}
The map $Pol$ is such named because it is often called Polarization map in the literature. Therefore, let us analyze the non-commutative cross-ratio from \cite{GGRW2005,Ze2006} which algebraic aspects are refreshed in \cite{RRS2020}. We recall its definition along the lines of our notations, and both extend and complete the statements which are, to our opinion, a little too fastly stated in \cite{RRS2020,Ze2006}, or at least, developed in a too reductive setting.

\begin{Definition}
	Let $((V_1,W_1),(V_2,W_2)) \in \mathbf{Spl}(A)^2$ such that $((V_1,W_2),(V_2,W_1)) \in \mathbf{Spl}(A)^2.$ We call such two splittings an {\bf admissible couple of polarizations} or {\bf admissible couple of splittings}. We denote by $Adm(V)$ the set of admissible couples of splittings. Then we define on $Adm(V)$  the map $DV$ with values in $L(V_1)$ defined by $$DV(V_1,W_1,V_2,W_2) = (Im \times Ker)^{-1}(V_1,W_2) \circ (Im \times Ker)^{-1}(V_2,W_1)|_{V_1}.$$
\end{Definition}

  \begin{remark} \label{rem:isom}
  This definition corresponds to the composition diagram $$ V_1 \xrightarrow[]{W_1} V_2 \xrightarrow[]{W_2} V_1$$ present in \cite{RRS2020,Ze2006}. Comparing our statement to those present in \cite{RRS2020,Ze2006}, the assumption $((V_1,W_2),(V_2,W_1)) \in \mathbf{Spl}(A)^2$ replaces the existence of a fixed isomorphism $\psi  \in GL(V_1,W_1)$ which is not a priori necessary at this step. Therefore, our construction can be understood as a tentative of generalization of the setting described \cite{RRS2020,Ze2006} even if we are not sure that this is strictly speaking a ``generalization'', but maybe only an ``adaptation''.
Let us now give more comments on the assumptions on $(V_1,W_1)$ and $(V_2,W_2)$ that one can find in \cite{RRS2020,Ze2006}. The assumption of isomorphism between $V_i$ and $W_i$, for $i \in \{1,2\},$ is not assumed here. However, we can check that the condition $((V_1,W_1),(V_2,W_2)) \in Adm(V),$ the diffeological vector spaces $V_1$ and $V_2$ are isomorphic. For this, we note by $p_{(V,W)}$ the projection on $V$ parallel to $W,$ with $(V,W) \in \mathbf{Spl}(A).$ Indeed, given $x_1 \in V_1$ and $x \in \pi_{(V_1,W_1)}^{-1} (x_1),$ this is an easy exercise to check that $ x_2 = p_{(V_2,W_1)}(x)$ only depends on $x_1$ and that $p_{(V_1,W_1)}(x_2)=x_1.$ We can state a bit quickly that $ \left( p_{(V_2,W_1)} |_{V_1} \right)^{-1} = p_{(V_1,W_1)}|_{V_2}.$ Therefore, our construction on admissible couples of splittings only implies that $V_1$ and $V_2$ are isomorphic, and that $W_1$ and $W_2$ are also isomorphic.  
\end{remark}
 	We have the following:
 	\begin{Proposition} \label{dv-inverse}
 	The map $DV$ is a smooth $GL(V_1)-$valued map, and its inverse reads as
 	$$DV(V_1,W_1,V_2,W_2)^{-1} = DV(V_1,W_2,V_2,W_1).$$
 	\end{Proposition}
 \begin{proof}
 	 Considering the diagram
 	$$ V_1 \xrightarrow[]{W_2} V_2 \xrightarrow[]{W_1} V_1$$ (which is possible under our assumptions), we see that we construct the inverse of the map $DV(V_1,W_1,V_2,W_2).$ The smoothness properties follow from the smoothness of the elementary operations that we use in this construction.
 \end{proof}

\begin{remark}
	The set of admissible couples of splittings, as it is defined here, is stable under the action of $J_{(-)}$ which is a smooth map on $Adm(A)$.
\end{remark}

Let us equip the space of admissible polarizations with its subset diffeology inherited from its inclusion in $\mathbf{Spl}(A)^2.$
\begin{Definition}
	The Grassmanian on $A$ is the space 
	$$\mathbf{Gr}(A) = \left\{ V \subset A \, | \, \exists s \in \mathbf{S}(A), \, V = Im(s)\right\}.$$
	We equip $\mathbf{Gr}(A)$ with its push-forward diffeology $\p_\mathbf{Gr} = Im_* \p_L.$
\end{Definition}
Denoting $\pi_i$ the canonical projection on the $i-$th component of $\mathbf{Gr}(A)^2$ with $i \in \{1;2\},$ and denoting the same way the restriction of the projection maps to $\mathbf{Spl}(A),$ we have the following commutative diagram:

\begin{center}
\begin{tikzpicture}
\node (spl) at (-1,1) {$\mathbf{Spl}(A)$};
\node (gr2) at (3,1) {$\mathbf{Gr}(A)^2$};
\node (gr) at (1,-1) {$\mathbf{Gr}(A)$};
\draw[->,>=latex,] (spl) -- (gr2); 
\draw[->,>=latex,] (spl) -- (gr) node[pos=0.5,right,below]{$\pi_i$};
\draw[->,>=latex,] (gr2) -- (gr) node[pos=0.5,left,below]{$\pi_i$}; 
\end{tikzpicture} 
\end{center}
where the horizontal arrow is the canonical inclusion map. 
\subsection{Yet another involutions on $Adm(V).$}
Let us now analyze the maps $$\tilde J: ((V_1,W_1),(V_2,W_2)) \in Adm(V) \mapsto ((V_1,W_2),(V_2,W_1))$$
and $$J_{sw} : ((V_1,W_1),(V_2,W_2)) \in Adm(V) \mapsto ((V_2,W_2),(V_1,W_1))$$
\begin{Proposition}
The maps $\tilde J$ and $J_{sw}$ are smooth on $Adm(V)$ with values in $Adm(V).$ 
\end{Proposition}
\begin{proof}
    By definition, we already know that $Im \tilde J = Adm(V).$  The same way, it is straightforward to prove that if $ ((V_1,W_1),(V_2,W_2)) \in Adm(V),$ then $((V_2,W_2),(V_1,W_1)) \in Adm(V).$ Moreover, since the map 
    $$sw: (a,b) \in \mathbf{S}(A)^2 \rightarrow (b,a) \in \mathbf{S}(A)^2$$ is smooth, we get that $$J_{sw} = (Pol \times Pol) \circ sw \circ (Pol^{-1} \times Pol^{-1}) $$ is smooth. 
    Moreover, the diffeology on $Adm(V)$ is the subset diffeology inherited from $\mathbf{Gr}(A)^4.$ Therefore, $\tilde J$ is smooth. 
\end{proof}
\begin{remark}
We recall that $$\mathfrak{D}_4 = \left\{Id_{\N_4}, (12)(34), (13)(24), (14)(23)\right\}$$
is a subgroup of $\mathfrak{S}_4.$
    The map $J_{sw}$ corresponds indexwise to the action of $(13)(24) \in \mathfrak{D_4}.$ The same way, $J_{(-)}$ corresponds to $(12)(34)\in \mathfrak{D_4}.$ But  $\tilde{J}$ is not in $\mathfrak{D}_4$ but correspond to $(24) \in \mathfrak{S}_4 - \mathfrak{D}_4.$ 
\end{remark}

\subsection{Non-commutative cross-ratio generates $GL(V)-$valued mappings} \label{s:3.4}

\begin{Theorem}
	The non-commutative cross ratio $DV$ on the space of admissible couples $((V_1,W_1),(V_2,W_2))$ of polarizations  defines the map $$\Phi : ((V_1,W_1),(V_2,W_2))\mapsto   DV(V_1,W_1,V_2,W_2)\circ (Im \times Ker)^{-1}(V_1,W_1)$$
	which is a smooth map with values in $L(V,V_1)\subset L(V)$ and moreover, 
 if  $a =  \phi((V_1,W_1),(V_2,W_2)),$ we have $$Ker(a) \supset W_1 \, \wedge \, Im(a) \subset V_1 .$$
\end{Theorem}
\begin{remark}
    We have to precise that the statement of smoothness is two-fold in this theorem: for any \emph{fixed} admissible couple of polarizations $((V_1,W_1),(V_2,W_2))$, the map $\Phi ((V_1,W_1),(V_2,W_2)): V \rightarrow V$ is smooth \emph{and} the maps $\Phi ((V_1,W_1),(V_2,W_2)) \in L(V)$ depend smoothly on $((V_1,W_1),(V_2,W_2)),$ where smoothness in with respect of the diffeology on $Adm(A)$ in this second aspect.
\end{remark}
\begin{proof}
	Smoothness follows from the smoothness of the diffeomorphism $Im \times Ker,$ and the algebraic part is direct by construction.
	\end{proof}
Moreover, we can define and state: 
\begin{Definition}
	$Adm(V)$, as it is defined here, is stable under the action of $J_{(-)}$. Therefore, we can define a global non-commutative cross-ratio map $\tilde{DV}  $ on admissible couples of splittings 
	  by
	$$ \tilde{DV}((V_1,W_1),(V_2,W_2)) = \Phi((V_1,W_1),(V_2,W_2)) + \Phi(J_{(-)}(V_1,W_1),J_{(-)}(V_2,W_2)).  $$
\end{Definition}

\begin{remark} \label{cross-ratio-id}
    By direct computations on projections, we have:
    $$\forall ((V_1,W_1)(V_2,W_2))\in Adm(V), [(V_1 = V_2) \vee (W_1 = W_2)] \Rightarrow \tilde{DV}((V_1,W_1),(V_2,W_2))=Id_V.$$
\end{remark}

Let us give now an elementary example in which $\tilde{DV}((V_1,W_1),(V_2,W_2)) \neq Id_V.$
\begin{example}
    Let $V=\R^2,$ $$V_1 = \{ (x,y) \in \R^2 \, | \, y =0 \},$$
    $$V_2 = \{ (x,y) \in \R^2 \, | \, x = y \},$$
    $$W_1 = \{ (x,y) \in \R^2 \, | \, x = 0 \}$$
    $$\hbox{ and } W_2 = \{ (x,y) \in \R^2 \, | \, x + y =0\}.$$
    Then $\tilde{DV}((V_1,W_1),(V_2,W_2))(x,y)=(2x,2y),$ therefore $$\tilde{DV}((V_1,W_1),(V_2,W_2))=2 Id_{\R^2}.$$
    More generally, for $\theta \in \left(0;\frac{\pi}{2}\right)$ 
    with  $V=\R^2,$ $$V_1 = \{ (x,y) \in \R^2 \, | \, y =0 \},$$
    $$V_2 = \{ (x,y) \in \R^2 \, | \,  y=\tan(\theta) x \},$$
    $$W_1 = \{ (x,y) \in \R^2 \, | \, x = 0 \}$$
    $$\hbox{ and } W_2 = \{ (x,y) \in \R^2 \, | \, x + \tan(\theta)y =0\}.$$
    Then $$\tilde{DV}((V_1,W_1),(V_2,W_2))= \frac{1}{\cos^2(\theta)} Id_{\R^2}.$$
\end{example}
More complex examples will be given in the dedicated section. 
\begin{Lemma}
    Let $((V_1,W_1)(V_2,W_2))\in Adm(V).$
    $$\left(\tilde{DV}((V_1,W_1),(V_2,W_2))\right)^{-1} = \Phi \circ \tilde{J} + \Phi \circ \tilde{J} \circ J_{(-)},$$
    therefore the map $$((V_1,W_1)(V_2,W_2))\in Adm(V) \mapsto \left(\tilde{DV}((V_1,W_1),(V_2,W_2))\right)^{-1}$$ is smooth in $L(V)$.
\end{Lemma}

\begin{Theorem} \label{th:tildeDV}
	Let $((V_1,V_2),(W_1,W_2)) \in Adm(V).$
 \begin{itemize}
 \item if $s = Pol^{-1}(V_1,W_1),$ $\Phi((V_1,W_1),(V_2,W_2)) \in Im(\pi_s^N).$  
 \item If $((V_1,W_1),(V_2,W_2))$ is an admissible couple of splittings, $$Im \tilde{DV}((V_1,W_1),(V_2,W_2)) \in GL(V_1) \times GL(W_1) \subset GL(V)$$
 \item If $Adm(V)$ is endowed with its subset diffeology as a subset of $\mathbf{Spl}(V),$ then $$\tilde{DV}:Adm(V) \rightarrow GL(V)$$ is smooth.
\end{itemize}
\end{Theorem}
The proof is direct from previous constructions.
\begin{Theorem}
    Let $((V_1,W_1),(V_2,W_2))\in Adm(V).$ Then the map
    $\Xi((V_1,W_1),(V_2,W_2))$ defined by 
    $$\Xi((V_1,W_1),(V_2,W_2)) = \tilde{DV}((V_1,W_1),(V_2,W_2))\circ \tilde{DV}((V_1,W_2),(V_2,W_1)) \in GL(V)$$
    is in $GL(V)$ and satisfies: 
    $$\Xi((V_1,W_1),(V_2,W_2))|_{V_1} = Id_{V_1}.$$
    Moreover the map $$\Xi: Adm(V) \rightarrow GL(V)$$ is smooth.
\end{Theorem}
\begin{proof} By Theorem \ref{th:tildeDV},  $\tilde{DV} \in C^\infty(Adm(V),GL(V)),$ therefore we also have that $\Xi \in C^\infty(Adm(V),GL(V)),$ writing $\Xi(.) = \tilde{DV} (.) \circ \tilde{DV} (\tilde J(.)).$ 
    Let us now analyze 
    \begin{eqnarray*}
        \Xi((V_1,W_1),(V_2,W_2))|_{V_1} &=& \tilde{DV}((V_1,W_1),(V_2,W_2))\circ \tilde{DV}((V_1,W_2),(V_2,W_1))|_{V_1} \\
        & = &\tilde{DV}((V_1,W_1),(V_2,W_2))|_{V_1}\circ {DV}((V_1,W_2),(V_2,W_1)) \\
                & = &{DV}((V_1,W_1),(V_2,W_2))\circ {DV}((V_1,W_2),(V_2,W_1)) \\
                & = & Id_{V_1} \hbox{ by Proposition \ref{dv-inverse}.}
    \end{eqnarray*}
    \end{proof}
\subsection{From non-commutative cross-ratio to relative Grassmannians} \label{s:3.5}


Let $V_1 \in \mathbf{Gr}(A).$ Let us now define the following domain:

\begin{Definition}
    We define $$Gr_{V_1}(A) = \pi_1^{-1}(V_1) = \left\{ W_1 \in Gr(A) \, | \, (V_1,W_1) \in Spl(A) \right\}$$ and we call it \emph{the Grassmannian relative to} $V_1$, or $V_1-$\emph{Grassmannian} for short.
\end{Definition}
This is the main object studied in this section. Before that, we analyze the map $\Xi$ in the blockwise decomposition $V_1 \oplus W_1.$ We have that 
$$\Xi((V_1,W_1),(V_2,W_2)) = \left( \begin{array}{cc} Id_{V_1} & * \\ 0 & * \end{array} \right).$$
    Let us now analyze, for $w_2 \in W_2!$
    \begin{eqnarray*}
        \Xi((V_1,W_1),(V_2,W_2))|_{W_2} &=& \tilde{DV}((V_1,W_1),(V_2,W_2))\circ \tilde{DV}((V_1,W_2),(V_2,W_1))|_{W_2} \\
        & = &\tilde{DV}((V_1,W_1),(V_2,W_2))|_{W_2}\circ {DV}((W_2,V_1),(W_1,V_2)) .
                \end{eqnarray*}
                Let $w'_2 = {DV}((W_2,V_1),(W_1,V_2))(w_2).$ In the decomposition $V=V_1 \oplus W_1,$ $$w'_2 = (w'_2)_{V_1} + (w'_2)_{W_1}, $$
and \begin{eqnarray*}\tilde{DV}((V_1,W_1),(V_2,W_2))(w'_2) & = &{DV}((V_1,W_1),(V_2,W_2))((w'_2)_{V_1}) \\ && + {DV}((W_1,V_1),(W_2,V_2))((w'_2)_{W_1}).\end{eqnarray*}
Since ${DV}((W_2,V_1),(W_1,V_2)) \in GL(W_2)$ and ${DV}((V_1,W_1),(V_2,W_2)) \in GL(V_1),$ the first term of the decomposition ${DV}((V_1,W_1),(V_2,W_2))((w'_2)_{V_1})$ vanishes identically if and only if $$\forall w_2 \in W_2, (w'_2)_{V_1} = 0 \Leftrightarrow W_1=W_2.$$

\begin{Theorem}
    Let $V_1\in \mathbf{Gr}(A).$ Then, $\forall (W_1,V_2,W_2) \in \mathbf{Gr}(A)$ such that $((V_1,W_1),(V_2,W_2)) \in Adm(A),$ then the map $\Xi((V_1,W_1),(V_2,W_2))$ defines a diffeomorphism on $\mathbf{Gr}_{V_1}(A)$ defined by $$W \in \mathbf{Gr}_{V_1}(A) \mapsto \Xi((V_1,W_1),(V_2,W_2))(W).$$
    Moreover, setting $$D_{V_1} = \left\{(W_1,V_2,W_2)\in \mathbf{Gr}(A) \, | \, ((V_1,W_1),(V_2,W_2))\in Adm(A) \right\},$$
    The map $$\Xi: (W_1,V_2,W_2) \in D_{V_1} \mapsto \Xi((V_1,W_1),(V_2,W_2)) \in Diff(\mathbf{Gr}_{V_1}(A))$$
    is smooth.
\end{Theorem}

\begin{proof}
    For $((V_1,W_1),(V_2,W_2)) \in Adm(A),$  $$\left\{ \begin{array}{ccl}
        \Xi((V_1,W_1),(V_2,W_2))(V_1) & = & V_1  \\
         \Xi((V_1,W_1),(V_2,W_2))(W) & = & W'
    \end{array} \right.$$ 
    and since the map $\Xi((V_1,W_1),(V_2,W_2))$ is in $GL(V),$
    $$V=\Xi((V_1,W_1),(V_2,W_2))(V_1 \oplus W)= V_1 + W' = V_1 \oplus W'.$$
    Therefore, $(V_1,W') \in \mathbf{Spl}(A)$ and $\Xi((V_1,W_1),(V_2,W_2))$ defines a map from $\mathbf{Gr}_{V_1}(A)$ to itself.
    Moreover, 
    $$\left(\Xi((V_1,W_1),(V_2,W_2))\right)^{-1} = \left(\tilde{DV}((V_1,W_2),(V_2,W_1))\right)^{-1} \circ \left(\tilde{DV}((V_1,W_1),(V_2,W_2)) \right)^{-1}$$
    and 
    $$ (V_1,W') = Pol^{-1} \left( \Xi((V_1,W_1),(V_2,W_2)) \circ Pol(V_1,W) \circ \left(\Xi((V_1,W_1),(V_2,W_2))\right)^{-1} \right) $$
   which proves simultaneously that that the map $\Xi((V_1,W_1),(V_2,W_2))$ from $\mathbf{Gr}_{V_1}(A)$ to itself is smooth, that the map  $\Xi: (W_1,V_2,W_2) \in D_{V_1} \mapsto \Xi((V_1,W_1),(V_2,W_2))$ is smooth and that $$\Xi((V_1,W_1),(V_2,W_2))^{-1} : W' \mapsto $$ $$\pi_2 \circ Pol^{-1} \left( \Xi((V_1,W_1),(V_2,W_2)) \circ Pol(V_1,W') \circ \left(\Xi((V_1,W_1),(V_2,W_2))\right)^{-1} \right) $$ is smooth by smoothness of each operation in these formulas.
\end{proof}
 
\section{Selected examples} \label{s:4}

\subsection{On two classical polarizations in quantum field theory} \label{s:4.1}
We describe here the splittings derived from the spectrum of a Dirac operator. The main reference for a classical description of these objects is \cite{PS}. 
\subsubsection{On the polarizations of $Map(S^1,\mathbb{C})$}

Lett us start with the complex Hilbert space $L^2(S^1,\mathbb{C}),$ with $S^1 = \R/2\pi \Z = \{z \in \mathbb{C} \, | \, |z|=1\}.$ The Fourier series are based on the orthonormal base $(z^n)_{n \in \Z},$ (where $z = e^{ix}$) defined by the eigenvectors of the (positive) Laplacian $\Delta = - \frac{\partial^2}{\partial x^2}.$ The associated Dirac operator $D = -i \frac{\partial}{\partial x}$ splits $E=L^2(S^1,\mathbb{C})$ into three spaces:
\begin{itemize}
	\item $E_+ = Span(z^n \, | \, n >0)$ associated to the positive eigenvalues of $D,$
	\item $E_- = Span(z^n \, | \, n <0)$ associated to the negative eigenvalues of $D,$
	\item $E_0 \sim \mathbb{C}$ the space of constant maps, which is the eigenspace associated to the eigenvalue $0$ of $D.$
\end{itemize}
This decomposition is $L^2-$orthogonal. In the sequel, we note by $\pi_+,\pi_-$ and $\pi_0$ the $L^2-$orthogonal projections on $E_+,$ $E_-$ and $E_0$ respectively.
\subsubsection{An admissible pair with respect to $E_0$ and its non-commutative cross-ratio}
The splitting $(E_+,E_-)$ is a realization of the standard splitting of $L^2(S^1,\mathbb{C})/E_0.$ Following e.g. \cite{Freed1988}, two realizations of  $L^2(S^1,\mathbb{C})/E_0$ in $L^2(S^1,\mathbb{C})$ can be highlighted:
\begin{enumerate}
    \item $H_1$ is the closure of $ \left\{ f \in C^\infty(S^1,\mathbb{C}) \, | \, f (0)=0 \right\} $
    \item $H_2 = \left\{ f \in L^2(S^1,\mathbb{C}) \, | \, \int_{S^1} f =0 \right\}=E_- \oplus E_+ .$ The condition $\int_{S^1} f =0$ is equivalent to $\pi_0(f)=0.$ 
\end{enumerate}
Next proposition is straightforward:
\begin{Proposition}
    $((E_0,H_1),(E_0,H_2))$ is an admissible pair.
\end{Proposition}
Let $i \in \N_2.$ We note by $p_i$ the projection on $H_i$ with respect to $E_0.$ Then $q_i= Id_E - p_i$ the projection on $E_0$ with respect to $H_i.$ Let $s_i = p_i - q_i.$ Then, 
\begin{Proposition}
    $$\tilde{DV}((E_0,H_1),(E_0,H_2))=Id_E.$$
\end{Proposition}
\begin{proof} First, since the two projections on $H_1$ and $H_2$ are both with respect to $E_0,$ we have that ${DV}((E_0,H_1),(E_0,H_2))=p_1 \circ p_2|_{H_1} = Id_{H_1}.$ 
\end{proof}
\subsubsection{Restricted Grassmannians, admissible pairs and non-commutative cross-ratio}
We \emph{choose} to set
\begin{itemize}
    \item $H_+=E_+ \oplus E_0$
    \item $H_- = E_-.$
\end{itemize}
All what follows holds true for other choices of splittings $H= L^2(S^1,\mathbb{C})= H_+ \oplus H_-$ with $E_\pm \subset H_\pm.$
We recall that, with our notations, if $\epsilon = Pol^{-1}(H_+,H_-),$ the Pressley-Segal-Wilson restricted linear group is the group $$GL_{res} = \left\{ g \in GL(H) \, | \, [\epsilon,g] \in \tau_2 \right\}$$
where $\tau_2$ is the ideal of $L(H)$ made of Hilbert-Schmidt operators
$$\tau_2 = \left\{a \in L(H) \, | \, ||a||_2 = tr(aa^*) = \sum_{n \in \Z} \left< a(z^n), a(z^n) \right>_H < +\infty \right\}.$$
The group $GL_{res}$ is the group of the units of the algebra $$L_{res} = \left\{ a \in L(H) \, | \, [\epsilon,a] \in \tau_2\right\}$$
equipped with the norm $$||a|| = ||a||_{L(H)} + ||[\epsilon,a]||_{2}.$$
As a group of the units of a Banach algebra, $GL_{res}$ is a Banach Lie group open in $L_{res}$ which implies that $$\p^{(-1)}_\infty(GL_{res}) = \p_\infty(GL_{res}).$$
We then work with:
$$A = L_{res}, \quad A^*= GL_{res}$$
and
$$\mathbf{S}(A) = \left\{ g \in GL_{res} \, | \, g^2=Id_H\right\}.$$
We decompose an operator $a \in L_{res}$ blockwise with respect to $(H_+,H_-)=Pol(\epsilon),$
traditionally written as 
$$a = \left( \begin{array}{cc}
a_{++} & a_{+-} \\
a_{-+} & a_{--}
\end{array} \right).$$
\begin{Proposition}
   Let $$a = \left( \begin{array}{cc}
a_{++} & a_{+-} \\
a_{-+} & a_{--}
\end{array} \right) \in GL_{res}.$$ Then $$\pi_\epsilon^T \left( \begin{array}{cc}
a_{++} & a_{+-} \\
a_{-+} & a_{--}
\end{array} \right) = \left( \begin{array}{cc}
0 & a_{+-} \\
a_{-+} & 0
\end{array} \right) \in \tau_2(H)$$
and 
$$\pi_\epsilon^N \left( \begin{array}{cc}
a_{++} & a_{+-} \\
a_{-+} & a_{--}
\end{array} \right) = \left( \begin{array}{cc}
a_{++} & 0 \\
0 & a_{--}
\end{array} \right) \in GL(H_+) \times GL(H_-).$$
\end{Proposition}
Let us now concentrate on the splitting $(H_+,H_-)$ and the corresponding splittings  $(V,W) \in \mathbf{Spl}(L_{res})$ such $ ((H_+,H_-),(V,W)) \in Adm(L_{res}).$
This last condition requires that $W \cap H_+ = \{0\}$ and that $V \cap H_- = \{0\}. $
\begin{Lemma}
Let $(V,W) \in \mathbf{Spl}(L_{res}),$ and let $s = Pol^{-1}(V,W).$ Then $$ ((H_+,H_-),(V,W)) \in Adm(L_{res}) \Leftrightarrow Ker(s_{-+}) = Ker(s_{-+})=\{0\}.$$
\end{Lemma}
\begin{proof}
    The condition $W \cap H_+ = \{0\}$ is equivalent to $$\forall x_+ \in H_+, \frac{1}{2}(x_+ + s(x_+))=x_+ \Leftrightarrow x_+=0.$$
    We have that $$\frac{1}{2}(x_+ + s(x_+)) = \frac{1}{2}(x_+ + s_{++}(x_+)) + \frac{1}{2} s_{-+}(x_+).$$
    We remark that $\frac{1}{2} s_{-+}(x_+) \in H_-,$ which implies that $s_{-+}$ must be injective. The same condition holds for $s_{+-},$ and the reciprocal statement is trivial.
\end{proof}
By Remark \ref{rem:isom}, if $((H_+,H_-),(V,W)) \in Adm(L_{res}),$ the spaces $H_+$ and $V$ are isomorphic, as well as the spaces $H_-$ and $W.$ Let $g \in GL(H)$ (in the diffeological sense) which satisfies $(g(H_+),g(H_-))=(V,W),$ then $g$ must satisfy the following system: 
\begin{eqnarray*} 
&& \left\{ \begin{array}{lcr} 
\frac{1}{2}\left(Id_H + s\right) \circ g \circ Id_{++} &=& g \circ Id_{++} \\
\frac{1}{2}\left(Id_H - s\right) \circ g \circ Id_{--} &=& g \circ Id_{--} \\
\end{array} \right. \\
& \Leftrightarrow &  (s - Id_{H}) \circ g \circ \epsilon = 0 \\
& \Leftrightarrow & \left\{ \begin{array}{lcr} 
 s_{++} g_{++} + s_{+-} g_{-+} + s_{-+}g_{++} + s_{--}g_{-+} &=& g_{++} + g_{-+} \\
s_{++} g_{+-} + s_{+-} g_{--} + s_{-+}g_{+-} + s_{--}g_{--} &=& -g_{--} - g_{+-} \\
\end{array} \right. \\
& \Leftrightarrow & \left\{ \begin{array}{lcr} 
g_{++} & = & s_{++} g_{++} + s_{+-} g_{-+} \\
g_{+-} & = & - s_{++} g_{+-} - s_{+-} g_{--} \\
g_{-+} & = & s_{-+}g_{++} + s_{--}g_{-+} \\
g_{--} & = & -s_{-+}g_{+-} - s_{--}g_{--} \\
\end{array} \right. \\
\end{eqnarray*}
Such an equation is a linear equation on spaces of linear operators $L(H),$ understood as a $L_{res}-$module, for which the study is not in the scope of this paper.  

\subsection{On spaces of pseudo-differential operators}

	We denote by $DO^k(S^1)$,$k \geq 0$, the differential operators of order less or equal than $k$.
	The algebra $DO(E)$ is {filtered  by the} order. It is a subalgebra of the algebra of classical pseudo-differential operators $Cl(S^1,V)$ that we describe shortly hereafter, focusing on its necessary aspects.
	This is an algebra that contains, for example, the square root of the Laplacian \begin{equation} \label{eq:integral}|D| = \Delta^{1/2} = \int_{\Gamma} \lambda^{1/2}(\Delta-\lambda Id)^{-1} d\lambda,\end{equation}
	where $\Delta = - \frac{d^2}{dx^2}$ is the positive Laplacian and $\Gamma$ is a contour around the spectrum of the Laplacian, see e.g. \cite{See,PayBook} for an exposition on contour integrals of pseudo-differential operators. $Cl(S^1,V)$ contains also the inverse of $Id+\Delta,$ and all
	smoothing operators on $L^2(S^1,V). $ Among smoothing operators one can find the heat operator 
	$$e^{-\Delta} = \int_{\Gamma} e^{-\lambda}(\Delta-\lambda Id)^{-1} d\lambda.$$ 
	pseudo-differential operators (maybe non-scalar) are linear operators acting on $C^\infty(S^1,V)$ which reads locally as
	$$ A(f) = \int e^{ix.\xi}\sigma(x,\xi) \hat{f}(\xi) d\xi$$ where $\sigma \in C^\infty(T^*S^1, M_n(\mathbb{C}))$ satisfying additional estimates on its partial derivatives {and $\hat{f}$ means the Fourier transform of $f$}. Basic facts on pseudo-differential operators defined 
	on a vector bundle $E \rightarrow S^1$ can be found e.g. in \cite{Gil}.
	
{\begin{remark} Since $V$ is finite dimensional, there exists $n \in \N^*$ such that $V \sim \C^n.$ Through this identification, a pseudo-differential operator $A \in Cl^(S^1,V)$ can be identified with a matrix $(A_{i,j})_{(i,j)\in \N_n^2}$ with coefficients $$A_{i,j} \in Cl(S^1,\mathbb{C}).$$ In other words, the identification $V \sim \mathbb{C}^n$ that we fix induces the isomorphism of algebras
 $$Cl(S^1,V) \sim M_n(Cl(S^1,\mathbb{C})).$$
 This identification will remain true and useful in the successive constructions below, and will be recalled if appropriate. When it will not carry any ambiguity, we will use the notation $DO(S^1),$ $Cl(S^1),$ etc. instead of $DO(S^1,\mathbb{C}),$ $Cl(S^1,\mathbb{C}),$ etc. for operators acting on the space of smooth functions from $S^1$ to $\mathbb{C}.$
\end{remark}}
	 {Pseudo-differential operators can be also described by their kernel $$K(x,y) = \int_{\R} e^{i(x-y)\xi} \sigma(x,\xi)d\xi$$ which is off-diagonal smooth.} Pseudo-differential operators {with infinitely smooth kernel (or "smoothing" operators)}, i.e. that are maps: $L^2 \rightarrow C^\infty$ form a two-sided ideal that we note by $Cl^{-\infty}(S^1,V).$ Their symbols  are those which are in the Schwartz space $\mathcal{S}(T^*S^1, M_n(\mathbb{C})).$ Remarkably, we have that $Cl^0 \subset L_{res}$ and, for $a \in Cl^0,$ following the notations previously defined, $a_{+-} \in Cl^{-\infty}$ and $a_{-+} \in Cl^{-\infty}$ following \cite{Ma2003,Ma2006-2}. The quotient 
	$\F Cl(S^1,V)=Cl(S^1,V)/Cl^{-\infty}(S^1,V)$ of the algebra of pseudo-differential operators by $Cl^{-\infty}(S^1,V)$ forms the algebra of formal pseudo-differential operators. 
	 {Another algebra, which is actually known as a subalgebra of $\F Cl(S^1,V)$ following \cite{MR2018}, is also called algebra of formal pseudo-differential operators. This algebra is generated by formal Laurent series $$ \Psi DO(S^1,V) = C^\infty(S^1,V)((\partial^{-1}))=\bigcup_{d \in \Z}\left\{ \sum_{k \leq d} a_k \partial^k \right\}$$ where each $a_k \in C^\infty(S^1,M_n(\mathbb{C}))$ and $\partial = \frac{d}{dx}.$ Let us precise hereafter a short but complete description of basic correspondence between $ \Psi DO(S^1,V)$ and $\F Cl(S^1,V)$.}
	
	Symbols $\sigma$ project to formal symbols and there is an isomorphism between formal pseudo-differential operators and formal symbols. A detailed study can be found in \cite[Tome VII]{Dieu}.  Classical pseudo-differential operators are operators $A$ which associated formal symbol $\sigma(A)$ reads as an asymptotic expansion $$\sigma(A)(x,\xi) \sim \sum_{k \in \Z, k \leq o} \sigma_k(A)(x,\xi)$$
	where the \textbf{partial symbol of order k} $$\sigma_k(A): (x,\xi)\in T^*S^1\setminus S^1 \mapsto \sigma_k(A)(x,\xi)\in M_n(\C)$$ is $k-$positively homogeneous in the $\xi-$variable, smooth on $T^*S^1\setminus S^1 = \{(x,\xi)\in T^*S^1 \, | \, \xi \neq 0\}$ and such that {$d\in \Z$} is the order of the operator $A.$ The order of a smoothing operator {we put} equal to $-\infty$ and the formal symbol of a smoothing operator is $0.$

	 The set $\mathcal{F}Cl(S^1,V)$ is not the same as the space of formal operators $ \Psi DO(S^1,V) $ which naturally arises  in the algebraic theory of PDEs, see e.g. \cite{KW} for an overview, but here the partial symbols $\sigma_k(A)$ of $A \in \Psi DO(S^1,V)$ are $k-$homogeneous. By the way one only has $\Psi DO(S^1,V) \subset \mathcal{F}Cl(S^1,V).$ 
	Two approaches for a global symbolic calculus {of }pseudo-differential operators have been described in \cite{BK,Wid}.

	\textbf{Notations.} 
	We shall denote note by
	$Cl^d(S^1,V)$ the vector space of classical pseudo-differential operators 
	of order{ $\leq d$}. We also denote by $Cl^{*}(S^1,V)$ the group of  {invertible in  $Cl(S^1,V)$ operators}.
	We denote the sets of formal operators adding the script $\mathcal{F}$. The algebra of formal pseudo-differential operators, identified wih formal symbols, is noted by ${\mathcal F}Cl^{}(S^1,V),$ and its group of invertible element is ${\mathcal F}Cl^{*}(S^1,V),$ while formal pseudo-differential operators of order less or equal to  $d \in \Z$ is noted by ${\mathcal F}Cl^{d}(S^1,V).$
	
	Through identification of $\F Cl(S^1,V)$ with the corresponding space of formal symbols, the space $\F Cl(S^1,V)$ is equipped with the natural locally convex topology inherited from the space of formal symbols. 
	A formal symbol $\sigma_k$ is a smooth function in  $C^\infty(T^*S^1\setminus S^1, M_n(\mathbb{C}))$ which is $k-$ positively homogeneous (i.e. homogeneous for $k>0)$), and hence it can be identified with an element of $C^\infty(S^1, M_n(\mathbb{C}))^2$ evaluating 
	$\sigma_k$ at $\xi = 1$ and $\xi = -1.$ Identifyting $\F Cl^d(S^1,V)$ with 
	$$ \prod_{k \leq d} C^\infty(S^1, M_n(\mathbb{C}))^2, $$ 
	the vector space $\F Cl^d(S^1,V)$ is a Fr\'echet space, 
	and hence $$\F Cl(S^1,V) = \cup_{d \in \Z} \F Cl^d(S^1,V)$$ 
	is a locally convex topological algebra. For our needs, it will be useful to consider it as a diffeological vector space (equiped with the nebulae diffeology).
	\subsubsection{$Cl$- and $PDO-$ Grassmanians}
	Let us now consider $F = C^\infty(S^1,\mathbb{C}$ and 
 $$ F_{+} = F \cap H_{+}, \quad F_{-} = F \cap H_{-}.$$
 Because the Fourier base is made of smooth functions, the decomposition $F = F_+ \oplus F_-$ holds true topologically, and the full space $Cl(S^1,\mathbb{C})$ of classical, maybe unbounded, pseudo-differential operators is acting smoothly on $F.$ The same holds for $PDO(S^1,\mathbb{C}).$ Therefore, the same constructions of $\mathbf{Gr}(A),$ $\mathbf{Spl}(A)$ for $A = Cl(S^1, \mathbb{C})$ and for $A = PDO(S^1, \mathbb{C})$ hold true.
\subsubsection{Splittings on spaces of formal pseudo-differential operators} \label{s:4.2.2}
Now we will introduce our main classes of classical pseudo-differential operators.
We define the operator $s_{KV}$ on $\F Cl(M,V)$ which extends the operator $ T^*M \rightarrow T^*S^1$ defined by $(x,\xi) \mapsto (x,-\xi)$ by $$ s : \sum_n \sigma_{n}(x,\xi) \mapsto \sum_n  (-1)^n\sigma_{n}\left(s(x,\xi)\right).$$ 
This operator obviously satisfies $s^2_{KV} = Id.$ The same way, let us recall that $\sigma(\epsilon(D))= {\xi}.|\xi|^{-1}.$  We define $s_D: \sum_n \sigma_{n}(x,\xi) \mapsto {\xi}.|\xi|^{-1}.\sum_n  \sigma_{n}\left(x,\xi\right).$ 

\begin{Definition} \label{d7} 
	A formal classical pseudo-differential operator $A$ on $E$ is called 
	\begin{itemize}
		\item {\bf odd 
			class} if and only if for all $n \in \Z$ and all $(x,\xi) \in T^*M$ we have:
		$$ \sigma_n(A) (x,-\xi) = (-1)^n  \sigma_n(A) (x,\xi) \Leftrightarrow s_{KV}(A)=A. $$
		We note it as $A \in \mathcal{F}Cl_{odd}(S^1,V).$ 
		\item  {\bf even 
			class} if and only if for all $n \in \Z$ and all 
		$(x,\xi) \in T^*M$ we have:
		$$ \sigma_n(A) (x,-\xi) = (-1)^{n+1}  \sigma_n(A) (x,\xi)\Leftrightarrow s_{KV}(A)=-A..$$ We note it as $A \in \mathcal{F}Cl_{even}(S^1,V).$
  \item We define $\mathcal{F}Cl_{+}(S^1,V) = Ker(Id - s_D)$ and $\mathcal{F}Cl_{+}(S^1,V) = Ker(Id + s_D).$
	\end{itemize} 
\end{Definition}

Odd class pseudo-differential operators were introduced in \cite{KV1,KV2}; they are called 
``even-even pseudo-differential operators'' in the treatise \cite{Scott}. 
In what follows, the subscript ${}_{odd}$ 
(resp. ${}_{even}$) is used in an obvious way. 
We have the
splitting 
$$ \mathcal{F}Cl(S^1,V) = \mathcal{F}Cl_{odd}(S^1,V) \oplus \mathcal{F}Cl_{even}(S^1,V)\; .$$
$\mathcal{F}Cl(S^1,V)$ also splits into the two ideals $\mathcal{F}Cl(S^1,V) = \mathcal{F}Cl_+(S^1,V) \oplus \mathcal{F}Cl_-(S^1,V)$. 
	This decomposition is explicit in \cite[section 4.4., p. 216]{Ka}, and we give an explicit description here following \cite{Ma2003,Ma2006-2}. 
	
				Similar to this identification, following \cite{MRu2021-1}, we have other { injections} { for } $\lambda \in \R^*:$ 
				
				$$\Phi_{\lambda,\mu} : \sum_{k \in \Z} a_{k} \left(\frac{d}{dx}\right)^k \in \Psi DO(S^1,V) \mapsto \sum_{k \in \Z} a_{k} \left(\lambda^k\left(\frac{d}{dx}\right)_+^k + \mu^k\left(\frac{d}{dx}\right)_-^k\right)  \in \F Cl(S^1,V) $$
				for {$(\lambda,\mu) \in \C^2 \backslash \{(0;0)\},$} with unusual convention $0^k = 0$ $\forall k \in \Z.$ Moreover, changing a little the notations of \cite{MRu2021-1}, we define for $\lambda \in \mathbb{C}^*$ 
    $$\Phi_{\epsilon(D),\lambda} : \sum_{k \in \Z} a_{k} \left(\frac{d}{dx}\right)^k \in \Psi DO(S^1,V) \mapsto \sum_{k \in \Z} a_{k} \epsilon(D) \left(\lambda\frac{d}{dx}\right)^k \in \F Cl(S^1,V) $$ 
				\begin{remark}
					 We have the following identifications: 
      \begin{itemize}
      \item $Im \Phi_{1,1}  = \F Cl_{odd}(S^1,V),$ 
      \item $Im \Phi_{\epsilon,1} = \F Cl_{even}(S^1,V),$
	   \item $Im \Phi_{1,0} = \F Cl_+(S^1,V), $ 
    \item $Im \Phi_{0,1} = \F Cl_-(S^1,V) .$
    \end{itemize}
				\end{remark}
				
    A direct calculation shows the following: 
    \begin{Lemma}
        
        \begin{itemize}
            \item  $\left(\F Cl_{+}(S^1,V),Im \Phi_{\lambda,\mu}\right) \in \mathbf{Spl}(\F Cl(S^1,V)) $ if and only if $(\lambda,\mu) \in (\mathbb{C}^*)^2-\{(\lambda,0) \, | \, \lambda \in \mathbb{C}^* \}.$
            \item  $\left(\F Cl_{-}(S^1,V),Im \Phi_{\lambda,\mu}\right) \in \mathbf{Spl}(\F Cl(S^1,V)) $ if and only if $(\lambda,\mu) \in (\mathbb{C}^*)^2-\{(0,\mu) \, | \, \mu \in \mathbb{C}^* \}.$
            \item  $\left((\F Cl_{+}(S^1,V),\F Cl_{-}(S^1,V)), (Im \Phi_{\lambda,\lambda}, Im \Phi_{\epsilon(D),\lambda})\right) \in Adm(\F Cl(S^1,V)) $ for any $\lambda \in \mathbb{C}^*.$
        \end{itemize}

    \end{Lemma}
For $\lambda = 1,$ we get the admissible couple of splittings
$$\left((\F Cl_{+}(S^1,V),\F Cl_{-}(S^1,V)), (\F Cl_{odd}(S^1,V),\F Cl_{even}(S^1,V))\right)$$
and again a direct computation finds
\begin{Theorem}
    $$\tilde{DV}\left((\F Cl_{+}(S^1,V),\F Cl_{-}(S^1,V)), (\F Cl_{odd}(S^1,V),\F Cl_{even}(S^1,V))\right) = 2 Id_{\F Cl(S^1,V)}.$$
\end{Theorem}
    
		\subsection{On polarizations with respect to signed measures}
Let $X$ be a possibly infinite dimensional diffeological space. Let $T$ be a topology which is a sub-topology of the $D-$ topology on $X.$ Let $\sigma(T)$ be the Borel algebra generated by open subsets of the topology $T.$ Let us recall a key property, see e.g. \cite{Fi2012}:
\begin{Proposition}[Jordan decomposition theorem]
	Any finite signed measure $\mu$ on $\sigma(T)$ splits into $$\mu = \mu^+ - \mu^-,$$ where $\mu^+$ and $\mu^-$ are finite positive measures. 
\end{Proposition}
\noindent
With this result, one can heuristically say that $X$ can split into two subsets defined by $supp(\mu^+)$ and $supp(\mu^-).$ This decomposition is only heuristic in most cases since elementary topological arguments show that $supp(\mu^+)\cap supp(\mu^-) = \emptyset $ if and only if $supp(\mu^+)$ and $supp(\mu^-)$ are topologically disconnected, that is, they are open and closed subsets. Therefore, we have to define
\begin{Definition}
	Let $\mu$ be a finite signed measure on $X.$ Let $(X^+,X^-)$ be a partition of $X$ of measurable sets. We say that this is a $\mu-$partition if $$supp(\mu^\pm) \subset X^\pm.$$
\end{Definition}  
\noindent
With the pair $(X^+,X^-)$ at hand, we can define related symmetries and projections.
\begin{Definition}
	Let $V$ be a diffeological vector space. The $\mu-$partition $(X^+,X^-)$ space of (set theoretical) maps $V^X$ into two spaces that are the kernels $Ker(p_{X^+})$ and  $Ker(p_{X^-})$ of the projectors
	$$ p_{X^+} : f \in V^X \mapsto \mathbf{1}_{X^+}f$$ and 
		$$ p_{X^-} : f \in V^X \mapsto \mathbf{1}_{X^-}f,$$
		where $\mathbf{1}_{X^-}$ and $\mathbf{1}_{X^+}$ are the indicator maps
		$$ \mathbf{1}_{X^\pm}(x) = \left\{\begin{array}{ccl}
		1 & \hbox{ if } & x \in X^{\pm} \\
		0 & \hbox{ if } & x \notin X^{\pm}	
		\end{array}
			\right.$$ 
			We define $$s_{(X^+,X^-)} = 2 p_{X^+} - Id_{V^X}.$$
\end{Definition}
\begin{remark}
	The terminology of $\mu-$partition is justified by $\mu(f) = |\mu|\left(s_{(X^+,X^-)} (f)\right)$ for any bounded measurable function $f.$
\end{remark}
Let us first fix a $\mu-$partition $(X^+,X^-),$ and we define $$V = Im\left(p_{X^+}\right) = \left\{f \in V^X\, | \, \forall x \in X^-, f(x)=0 \right\}$$
and 
$$W = Im\left(p_{X^-}\right) = \left\{f \in V^X\, | \, \forall x \in X^+, f(x)=0 \right\}$$
 We have the following easy property:
 \begin{Proposition}
 	$V^X,$ as a cartesian product indexed by $V,$ is a diffeological vector space for the pointwise operations,
 $(V,W)\in \mathbf{Spl}(V^X)$ and	$(V,W) = Pol\left(s_{(X^+,X^-)}\right).$ 
 \end{Proposition} 
However, we have the following ``no-go'' properties. 
\begin{Theorem}
	\begin{itemize}
		\item Unless $X$ is equipped with the discrete diffeology or $V \neq \{0\},$ the action of $Diff(X)$ on $V^X$ by composition on the right is not smooth (in fact, not continuous for the underlying $D-$topologies.) 
		\item Given two different partitions $(X_1^+,X_1^-)$ and $(X_2^+,X_2^-)$ associated to two (possibly identical) signed measures,  the corresponding couple of splittings $((V_1,W_1),(V_2,W_2))$ is never admissible.
		\item The set of symetries of the type $s(X^+,X^-),$ defined for any partition $(X^+,X^-)$ of $X$ related to any signed finite measure, is totally disconnected in $\mathbf{S}(V^X).$
	\end{itemize}
\end{Theorem}
\begin{proof}
	If $V \neq \{0\},$ there exists $v \in V-\{0\}.$ Let $x \in X.$ We define the Dirac like function 
	 $\delta_x^v$ defined by 
	 $$\forall x' \in ,  \delta_x^v(x') = \left\{\begin{array}{ccl} v & \hbox{ if } & x'=x \\
	 	0 & \hbox{ if } & x' \neq x \end{array}\right.$$
 	If $X$ is not discrete, $Diff(X)$ is not discrete, and we choose a non constant path $g$ in $Diff(X)$ for which we can assume without restriction that $g(0)=Id_X.$ and we assume that we choose $x$ such that there exists $t \neq 0$ such that $g(t)(x) \neq x.$ Then, the smooth path $ev_x \circ \delta_x^v \circ g$ from $\R$ to $V$ takes only two values: $O$ and $v.$ This is impossible since the image of a connected set by a smooth map is connected. This proves the first point.   
	
	Let $(X_1^+,X_1^-)$ and $(X_2^+,X_2^-)$ be two different partitions of $X.$ Since the partitions are different, we have either $X_1^+ \cap X_2^- \neq \emptyset$ or $X_2^+ \cap X_1^- \neq \emptyset.$ Let $x \in X_1^+ \Delta X_2^+ = (X_1^+ \cap X_2^-) \cup (X_2^+ \cap X_1^-).$
	Then the Dirac function $\delta_x \in (V_1 \cap W_2) \cup (V_2 \cap W_1).$ Therefore, $V_1 \cap W_2 \neq \emptyset$ or $V_2 \cap W_1 \neq \emptyset,$ which shows that $((V_1,W_2),(V_2,W_1))\notin \mathbf{Spl}(V^X).$ Therefore,  $((V_1,W_1),(V_2,W_2))$ is not admissible.
	For the third point, we first assume that $V=\R.$
	Let us now assume that $s_1$ and $s_2$ are the two symmetries constructed from $(X_1^+,X_1^-)$ and $(X_2^+,X_2^-)$ and that there is a smooth path $c$ in the set of symetries of the type $s(X^+,X^-),$ defined for any partition $(X^+,X^-)$ of $X$ related to any signed finite measure, such that $c(0)=s_1$ and $c(1)=s_2.$ Let us consider now the evaluation map $ev_x.$ This map is smooth on $V^X,$ and hence the map $t \mapsto ev_x \circ c(t) \circ \delta_x$ is also smooth from $\R$ to $\R.$ But the image set of this map is $\{-1,1\}$ which is disconnected. Therefore the assumption is impossible. Passing from $\R^X$ to $V^X,$ we change $\delta_x$ to $\delta_x^v. $  
\end{proof}

\begin{remark}
Therefore, if one has to consider smooth structures on $V^X,$ one has to enrich or adapt the cartesian product diffeology. Let us make a review of few apparently interesting properties.
\begin{itemize}
	\item The cartesian product diffeology is the natural diffeology for the smoothness of the evaluation maps $ev_x.$ Indeed, it is the pull-back diffeology $\bigcap_{x \in X} ev^*\p_V.$
	\item The same way, let us consider the map $$\Phi : (f,g) \in V^X \times Diff(X) \rightarrow f \circ g \in V^X.$$ Then $\Phi(V^X,\p_{Diff(X)})$ generates the smallest diffeology for which the action of $Diff(X)$ on $V^X$ is smooth.   
\end{itemize} 
These two diffeologies are completely different.  
\end{remark}
These two points suggests that we need to restrict our space of functions in order to be interesting.
We now consider the space of \emph{smooth} functions $C^\infty(X,V)$ instead of $V^X.$ The space $C^\infty(X,V)$ is equipped with its functional diffeology. We hope to be able to expand this approach, envolving extra-structures, in a forthcoming work.		
	\section*{Acknowledgements}
 The author would like to thank Vladimir Rubtsov for informal introductions to his works on generalizations of cross-ratios to non-commutative settings, and for explanations on his algebraic motivations. 

 The starting point of this work is an idea that emerged during the XL workshop on methods in Goemetry and Physics in Bialowieza, july 2023, in which the author was invited for a plenary lecture. Helpful discussions principally with Daniel Beltit\`a, but also with Gerald Goldin, Tomas\'z Goli\'nski, Tudor Ratiu and Alice Barbara Tumpach, discussions on their own research interests, convinced the author about the accuracy of this topic. May they all be acknowledged, as well as the organizers of the workshop for this invitation and this very stimulating week.

\end{document}